\documentclass{amsart}

\usepackage{amssymb,hyperref}

\newtheorem{thm}{Theorem}
\newtheorem{lem}[thm]{Lemma}
\newtheorem{prop}[thm]{Proposition}
\newtheorem{cor}[thm]{Corollary}

\DeclareMathOperator{\End}{End}
\DeclareMathOperator{\Hom}{Hom}

\DeclareMathOperator{\Spec}{Spec}
\DeclareMathOperator{\M}{M}
\DeclareMathOperator{\rdim}{rdim}
\DeclareMathOperator{\SB}{SB}
\DeclareMathOperator{\ind}{ind}

\newcommand{\PP}{\mathbb{P}}
\newcommand{\Fp}{\mathbb{F}_p}
\newcommand{\Aa}{\mathbb{A}}

\author{Charles De Clercq}
\address[Charles De Clercq]{Universit\'{e} Sorbonne Paris Nord, Intitut Galil\'{e}e, Laboratoire Analyse, 
G\'{e}om\'{e}trie et Applications, Villetaneuse, France}
\email{declercq@math.univ-paris13.fr}
\urladdr{math.univ-paris13.fr/~declercq}

\author{Evan Marth}
\address[Evan Marth]{Department of Mathematics and Statistics, University of Ottawa, 150 Louis-Pasteur, Ottawa, ON, K1N 6N5, Canada}
\email{emarth@uottawa.ca}

\author{Kirill Zainoulline}
\address[Kirill Zainoulline]{Department of Mathematics and Statistics, University of Ottawa, 150 Louis-Pasteur, Ottawa, ON, K1N 6N5, Canada}
\email{kirill@uottawa.ca}
\urladdr{kirillmath.ca}


\begin{document}

\title{Rost nilpotence for twisted Milnor hypersurfaces}

\begin{abstract} 
We show that the strong Rost nilpotence (in the sense of S.~Gille) holds for motives of generic hyperplane sections of twisted Milnor hypersurfaces. 
Hence, we provide a new family of examples of smooth projective algebraic varieties which satisfy the strong Rost nilpotence principle. 
As an application, we compute the $p$-canonical dimension for such varieties.
\end{abstract}

\thanks{E.M. and K.Z. were partially supported by NSERC Discovery grant RGPIN-2022-03060}

\maketitle

These notes are essentially based on the results from four papers:  \cite{VZ08, XZ24, Ma25, Gi25} to which we also refer for most of the notations and definitions.

Let $V$ be a vector space of dimension $n+1$ over a field $k$. Consider the partial flag variety 
\[ 
E = \{ W_1 \subsetneq W_n \subsetneq V \colon \dim W_i = i \} 
\] 
Let $\varphi \in \End(V)$ be an endomorphism with $n+1$ distinct eigenvalues over the algebraic closure of $k$ such that the subalgebra $L := k[\varphi] \subseteq \End(V)$ is Galois  of dimension $n+1$ over $k$. Consider the hyperplane section of $E$ 
\[ 
Y = \{ W_1 \subsetneq W_n \subsetneq V \colon \dim W_i = i, \; \varphi(W_1) \subsetneq W_n \} 
\]
By the main result of \cite{Ma25} (Thm.1.1 loc.cit.), the Chow motive of $Y$ (with integer coefficients) decomposes as 
\begin{equation}\label{eq:an}
\M(Y) = \bigoplus_{i=0}^{n-2} \M(\PP^n_k)(i) \oplus \M(\Spec L)(n-1). 
\end{equation} 

Assume now that $\varphi$ is of the form $\varphi' \oplus diag(a_{m+1},\ldots,a_{n})$, where $L'=k[\varphi']$ is proper Galois of degree $m+1$ and $a_i\in k^\times$.
So $L$ is the product of $L'$ with $n-m$ copies of $k$. Then the computation of $\M(Y)$ of the respective hyperplane section can be reduced to that of its Galois part $\M(Y')$ using the following analogue of the Rost filtration for quadrics:

In projective coordinates 
$([x_0:\ldots : x_n],[y_0:\ldots:y_n]) \in \PP^{n}\times \PP^n$ the hyperplane section $Y=Y_{m,n}$ is given by the system of equations
\[
\begin{cases}
\sum_{i=0}^n x_iy_i=0\quad \text{ (this defines }E)\\
(x_0,\ldots,x_m)\varphi' (y_0,\ldots,y_m)^t +\sum_{i=m+1}^n a_ix_iy_i=0.
\end{cases}
\]
Define $Z$ to be the closed subset of $Y_{m,n}$ given by $x_n=0$. Consider the open complement $Y_{m,n}\setminus Z$ given by $x_n\neq 0$. It projects onto $\PP^{n-1}=\{[y_0:\ldots :y_{n-1}]\}$ with the fibre $\Aa^{n-1}$. Now let $Z'=\PP^{n-1}\times pt$ be the closed subset in $Z$ given by $y_n\neq 0$ and $y_i=0$, $i<n$. Then the map forgetting the last coordinates $x_n$ and $y_n$ defines the projection 
$Z \setminus Z' \to Y_{m,n-1}$ with the fibre $\Aa^{\!1}$.
The relative cellular decomposition of \cite[Cor.66.4]{EKM} then implies the motivic decomposition:
\begin{equation}\label{eq:iso}
\M(Y_{m,n})=\M(\PP^{n-1})(n-1)\oplus \M(Y_{m,n-1})(1)\oplus \M(\PP^{n-1}).
\end{equation}

Following \cite[Def.2.1.(b)]{Gi25} we say that the strong Rost Nilpotence principle holds for a Chow motive $M$ over $k$  if for all field extensions $l/k$, the kernel of the restriction morphism
\[
res_{l/k}\colon \End_k(M) \longrightarrow \End_l(M_l)
\]
is a nilpotent ideal. 
We say that a motive $M$ is Artin-Tate if
\[
M = M_1(i_1)\oplus \cdots \oplus M_r(i_r),\quad i_1<\ldots <i_r,
\] 
where the motives $M_i$ are direct summands of motives of $0$-dimensional varieties.  
\begin{lem}\label{lem:rna}
The strong Rost Nilpotence holds for Artin-Tate motives.
\end{lem}

\begin{proof}
Indeed, 
for any $\psi\in \End (M)$, we have $\psi=\psi_1\oplus\ldots \oplus\psi_r$, where each $\psi_i\in \End(M_i)$. Hence, we may assume $M$ is the motive of some 0-dimensional variety $X$. But the restriction morphism $res_{l/k}$ is known to be injective in this case (the Chow group of $X$ is free). 
\end{proof}

Combining \eqref{eq:an}, \eqref{eq:iso}, Lemma~\ref{lem:rna} and applying \cite[Thm.2.4]{Gi25} we obtain:

\begin{prop}
The strong Rost Nilpotence holds for the motive $\M(Y)$.
\end{prop}

We now twist the variety $Y$ by a 1-cocycle corresponding to the central simple algebra $A$ of degree $n+1$ so that
$L$ is its maximal commutative subalgebra. In other words, the variety $E_A$ is the partial flag of right ideals
\[
E_A = \{ I_1 \subsetneq I_n \subsetneq A : \rdim I_i = i \},
\]
and $Y_A$ is given by $\varphi I_1\subset I_n$
(if $A$ is split, $I_i=\Hom(V, W_i)$).
Observe that $Y_A$ is a generalization of the major example of \cite{XZ24}. 

\begin{lem}\label{generic}
Let $X$ be an irreducible smooth projective variety over $k$. Assume that $X$ satisfies the strong Rost nilpotence generically (that is the $k(X)$-variety $X_{k(X)}$ satisfies the strong Rost nilpotence). 
Then $X$ satisfies the strong Rost nilpotence.
\end{lem}

\begin{proof}
Following the proof of \cite[Prop.3.1]{VZ08} the restriction morphism  factors as the composite
\[
\End_k(\M(X))\xrightarrow{res_{X/k}} \End_{X}(\M(X_X))\xrightarrow{res_{k(X)/X}}\End_{k(X)}(\M(X_{k(X)}))
\]
where the middle term denotes endomorphisms in the category of Chow motives over the relative base $X$, the first map is injective, and the second map is induced by passing to the generic point of $X$. 
Observe that the kernel of $res_{k(X)/X}$ is the nilpotent ideal by the proof of \cite[Lem.3.2]{VZ08} extended in \cite[\S3.5]{Gi25}. 

Now the kernel of the restriction $res_{l/k}\colon \End_k(\M(X)) \to \End_l(\M(X_l))$ is contained in the kernel of $res_{F/k}$, where $F$ contains both $l$ and $k(X)$, so is nilpotent.
\end{proof}

\begin{thm}
The strong Rost nilpotence holds for the motive $\M(Y_A)$.
\end{thm}

\begin{proof}
From \cite[\S2]{XZ24} it follows that the motive $\M(Y_A)$ splits as
\[
\M(Y_A)=N\oplus \bigoplus_{i=0}^{n-3}\M(\SB(A))(i),
\] 
where the motive $N$ is supported only in the middle codimension (over the algebraic closure of $k$, $N$ splits as a direct sum of Tate motives twisted by $n-2$).

Since $\M(Y_A)$ is supported in dimension zero only by the motive of a Severi-Brauer variety $\M(\SB(A))$, the latter must contain a Tate motive (in dimension 0) over the generic point $k(Y_A)$. Therefore, by \cite{Ka96} the algebra $A$ splits over $k(Y_A)$. Hence the variety $Y_A$ turns over $k(Y_A)$ into the variety $Y$, and the strong Rost Nilpotence holds for $Y_A$ generically. To finish we apply Lemma~\ref{generic}.
\end{proof}

Since the variety $Y_A$ is geometrically split, and it satisfies the Rost Nilpotence principle, by \cite[Cor.2.6]{Ka13} the Krull-Schmidt principle holds for its motive with finite coefficients 
$\Fp=\mathbb{Z}/p\mathbb{Z}$, where $p$ is a prime. The decomposition of the motive of $\SB(A)$ appearing in \cite{Ka96} can then be refined as
\[
\M(\SB(A);\Fp)=
\bigoplus_{r=0}^{m} \M(\SB(D_p);\Fp)(rd)\;
\text{ where } d=p^{v_p(\ind(A))},\; m=\tfrac{\deg(A)}{d}-1,
\]
and $D_p$ is the $p$-primary component of a division algebra Brauer equivalent to $A$.

\begin{cor}
The canonical $p$-dimension of the variety $Y_A$  is equal to $d-1$.
\end{cor}

\begin{proof} 
Let $U$ be the upper motive of $\M(Y_A; \Fp)$. It is isomorphic to the indecomposable motive $\M(\SB(D_p);\Fp)$ up to a shift. Observe that $\dim(\SB(D_p))=d-1$. Being an indecomposable lower summand, the dual motive $U^*$ of $U$ is isomorphic to $U(\dim(Y_A)-d-1)$ by the Krull-Schmidt property (see \cite[Prop.5.2]{Ka10}). By \cite[Thm.5.1]{Ka10}, the canonical $p$-dimension of $Y_A$ is then equal to the dimension of $U$ that is $d-1$.
\end{proof}


\begin{thebibliography}{99}



\bibitem[EKM]{EKM}
Elman,~R.; Karpenko,~N.; Merkurjev,~A. 
The algebraic and geometric theory of quadratic forms. {\it AMS Colloquium Publications} {\bf 56}; 2008.

\bibitem[Gi25]{Gi25}
Gille,~S.
Direct sums of Chow motives and Rost Nilpotence.
Preprint 2025.

\bibitem[Ka13]{Ka13}
Karpenko,~N.
Upper motives of algebraic groups and incompressibility of Severi-Brauer varieties
{\it J. Reine Angew. Math.} 677 (2013), {179--198}.


\bibitem[Ka10]{Ka10}
Karpenko,~N.
Canonical dimension.
{\it Proceedings of the ICM 2010}, vol. II, {146--161}.


\bibitem[Ka96]{Ka96}
Karpenko,~N.   
The Grothendieck-Chow motifs of Severi-Brauer varieties. 
{\it St. Petersburg Math.~J.} 7 (1996), no.4, {649--661}.


\bibitem[Ma25]{Ma25}
Marth,~E.
Motives of certain Hyperplane Sections of Milnor Hypersurfaces.
{\it arXiv:2507.05434}.

\bibitem[VZ08]{VZ08} 
Vishik,~A.; Zainoulline,~K. 
Motivic Splitting Lemma. 
{\it Documenta Math.} 13 (2008), {81--96}.

\bibitem[XZ24]{XZ24}
Xiong,~R.; Zainoulline,~K.
Motivic Lefschetz theorem for twisted Milnor hypersurfaces.
{\it arxiv: 2404.07314}

\end{thebibliography}
\end{document}